\documentclass[12pt,a4paper]{article}
\usepackage{amssymb,amsmath,amsthm}
\usepackage{graphicx}

\newcommand\cA{{\mathcal A}}

\newcommand\cF{{\mathcal F}}
\newcommand\cG{{\mathcal G}}

\newcommand\cM{{\mathcal M}}

\newcommand\cP{{\mathcal P}}

\newcommand\cQ{{\mathcal Q}}
\newcommand\cR{{\mathcal R}}
\newcommand\cS{{\mathcal S}}
\newcommand\cT{{\mathcal T}}

\theoremstyle{plain}
\newtheorem{theorem}{Theorem}[section]
\newtheorem{lemma}[theorem]{Lemma}
\newtheorem{corollary}[theorem]{Corollary}
\newtheorem{conjecture}[theorem]{Conjecture}
\newtheorem{proposition}[theorem]{Proposition}

\theoremstyle{definition}

\newcommand\tref[1]{Theorem~\ref{thm:#1}}
\newcommand\cref[1]{Corollary~\ref{cor:#1}}

\newcommand\cjref[1]{Conjecture~\ref{conj:#1}}

\title{Generalized forbidden subposet problems}
\author{D\'aniel Gerbner\thanks{Research supported by the J\'anos Bolyai Research Fellowship of the Hungarian Academy of Sciences and the National Research, Development and Innovation Office -- NKFIH under the grant the grant K 116769.}, Bal\'azs Keszegh\thanks{Research supported by  the National Research, Development and Innovation Office -- NKFIH under the grant the grant K 116769.}, Bal\'azs Patk\'os\thanks{Research supported by the J\'anos Bolyai Research Fellowship of the Hungarian Academy of Sciences and the National Research, Development and Innovation Office -- NKFIH under the grants SNN 116095 and K 116769.} \\
\small Alfr\'ed R\'enyi Institute of Mathematics, Hungarian Academy of Sciences\\
\small P.O.B. 127, Budapest H-1364, Hungary.}

\begin{document}
\maketitle

\begin{abstract}
A subfamily $\{F_1,F_2,\dots,F_{|P|}\}\subseteq \cF$ of sets is a copy of a poset $P$ in $\cF$ if there exists a bijection $\phi:P\rightarrow \{F_1,F_2,\dots,F_{|P|}\}$ such that whenever $x \le_P x'$ holds, then so does $\phi(x)\subseteq \phi(x')$. For a family $\cF$ of sets, let $c(P,\cF)$ denote the number of copies of $P$ in $\cF$, and we say that $\cF$ is $P$-free if $c(P,\cF)=0$ holds. For any two posets $P,Q$ let us denote by $La(n,P,Q)$ the maximum number of copies of $Q$ over all $P$-free families $\cF \subseteq 2^{[n]}$, i.e. $\max\{c(Q,\cF): \cF \subseteq 2^{[n]}, c(P,\cF)=0 \}$. 

This generalizes the well-studied parameter $La(n,P)=La(n,P,P_1)$ where $P_1$ is the one element poset, i.e. $La(n,P)$ is the largest possible size of a $P$-free family. The quantity $La(n,P)$ has been determined (precisely or asymptotically) for many posets $P$, and in all known cases an asymptotically best construction can be obtained 
by taking as many middle levels as possible without creating a copy of $P$.

In this paper we consider the first instances of the problem of determining $La(n,P,Q)$. We find its value when $P$ and $Q$ are small posets, like chains, forks, the $N$ poset and diamonds. Already these special cases  show that the extremal families are completely different from those in the original $P$-free cases: sometimes not middle or consecutive levels maximize $La(n,P,Q)$ and sometimes no asymptotically extremal family is the union of levels.

Finally, we determine (up to a polynomial factor) the maximum number of copies of complete multi-level posets in $k$-Sperner families. The main tools for this are the profile polytope method and two extremal set system problems that are of independent interest: we maximize the number of $r$-tuples $A_1,A_2,\dots, A_r \in \cA$ over all antichains $\cA\subseteq 2^{[n]}$ such that  (i) $\cap_{i=1}^rA_i=\emptyset$, (ii) $\cap_{i=1}^rA_i=\emptyset$ and $\cup_{i=1}^rA_i=[n]$.
\end{abstract}

\section{Introduction}
The very first theorem of extremal finite set theory is due to Sperner \cite{S} and states that if $\cF$ is a family of subsets of $[n]=\{1,2,\dots,n\}$ such that no two sets in $\cF$ are in inclusion, then $|\cF|\le \binom{n}{\lceil n/2\rceil}$ holds, and equality is achieved if and only if $\cF$ consists of all the $\lceil n/2\rceil$-element or all the  $\lfloor n/2\rfloor$-element subsets of $[n]$. Families consisting of all the $k$-element subsets of $[n]$ are called \textit{(full) levels} and we introduce the notation $\binom{[n]}{k}=\{F\subseteq [n]:|F|=k\}$ for them. Sperner's theorem was generalized by Erd\H os \cite{Er} to the case when $\cF$ is not allowed to contain $k+1$ mutually inclusive sets, i.e. a $(k+1)$-chain. He showed that among such families the ones consisting of $k$ middle levels are the largest. In the early eighties, Katona and Tarj\'an \cite{KT} introduced a generalization of the problem and started to consider determining the size of the largest family of subsets of $[n]$ that does not contain a configuration defined by inclusions. Such problems are known as forbidden subposet problems and are widely studied (see the recent survey \cite{GLi}).

In this paper, we propose even further generalizations: we are interested in the maximum number of copies of a given configuration $Q$ in families that do not contain a forbidden subposet $P$. Before giving the precise definitions, let us mention that similar problems were studied by Alon and Shikhelman \cite{AS} in the context of graphs when they considered the problem of finding the most number of copies of a graph $T$ that an $H$-free graph can contain.

\vskip 0.3truecm

\noindent \textbf{Definition.} Let $P$ be an arbitrary poset and $\cF$ a family of sets. We say that $\cG\subseteq \cF$ is \textit{a copy of $P$ in $\cF$} if there exists a bijection $\phi:P\rightarrow \cG$ such that whenever $x \le_P x'$ holds, then so does $\phi(x)\subseteq \phi(x')$. Let $c(P,\cF)$ denote the number of copies of $P$ in $\cF$ and for any pair of posets $P,Q,$ let us define
\[
La(n,P,Q)=\max\{c(Q,\cF): \cF\subseteq 2^{[n]}, c(P,\cF)=0 \},
\]
and for families of posets $\cP,\cQ$ let us define
\[
La(n,\cP,\cQ)=\max\left\{\sum_{Q\in\cQ}c(Q,\cF):\cF\subseteq 2^{[n]}, \forall P\in\cP  \hskip 0.2truecm c(P,\cF)=0\right\}.
\]
We denote by $P_k$ the chain of length $k$, i.e. the completely ordered poset on $k$ elements. In particular, $P_1$ is the poset with one element.  
 Let us state Erd\H os's above mentioned result with our notation.

\begin{theorem}[Sperner \cite{S} for $k=1$, Erd\H os \cite{Er} for general $k$]
\label{thm:sperner}
For every positive integer $k$ the following holds:
$$La(n,P_{k+1}, P_1)=\sum_{i=1}^k\binom{n}{\lfloor\frac{n-k}{2}\rfloor+i}.$$
\end{theorem}

\vskip 0.2truecm
The area of forbidden subposet problems deals with determining $La(n,P)=La(n,P,P_1)$ the maximum size of a $P$-free family. There are not many results in the literature where other posets are counted. Katona \cite{K} determined the maximum number of 2-chains (copies of $P_2$) in a 2-Sperner ($P_3$-free) family $\cF\subseteq 2^{[n]}$. This was reproved in \cite{P} and generalized by Gerbner and Patk\'os in \cite{GP}.

\begin{theorem}[\cite{GP}]
\label{thm:gp}
For any $l> k$ the quantity $La(n,P_l,P_k)$ is attained for some family $\cF$ that is the union of $l-1$ levels. Moreover, $La(n,P_{k+1},P_k)=\binom{n}{i_k}\cdot \binom{i_k}{i_{k-1}} \cdot \dots \cdot
\binom{i_2}{i_1}$, where $i_1 <i_2<\dots<i_k<n$ are chosen arbitrarily such that the values $i_1,i_2-i_1,i_3-i_2,\dots,i_k-i_{k-1},n-i_k$ differ by at most one.
\end{theorem}

In this paper, we address the first non-chain instances of the general problem. We will consider the following posets (see Figure \ref{fig:defs}): let $\bigvee_r$ denote the poset on $r+1$ elements $0,a_1,a_2,\dots, a_r$ with $0\le a_i$ for all $i=1,2,\dots, r$ and we write $\bigvee$ for $\bigvee_2$. Similarly, let $\bigwedge_r$ denote the poset on $r+1$ elements $a_1,a_2,\dots, a_r,1$ with $a_i\le 1$ for all $i=1,2,\dots, r$ and we write $\bigwedge$ for $\bigwedge_2$. The poset $N$ contains four elements $a,b,c,d$ with $a\le c$ and $b\le c,d$. The butterfly poset $B$ consists of four elements $a,b,c,d$ with $a,b\le c,d$. 
Let the generalized diamond poset $D_k$ be the poset on $k+2$ elements $a,b_1,b_2,\dots,b_k,c$ with $a<b_1,b_2,\dots,b_k<c$. 

\begin{figure}[h]
\centering
\includegraphics{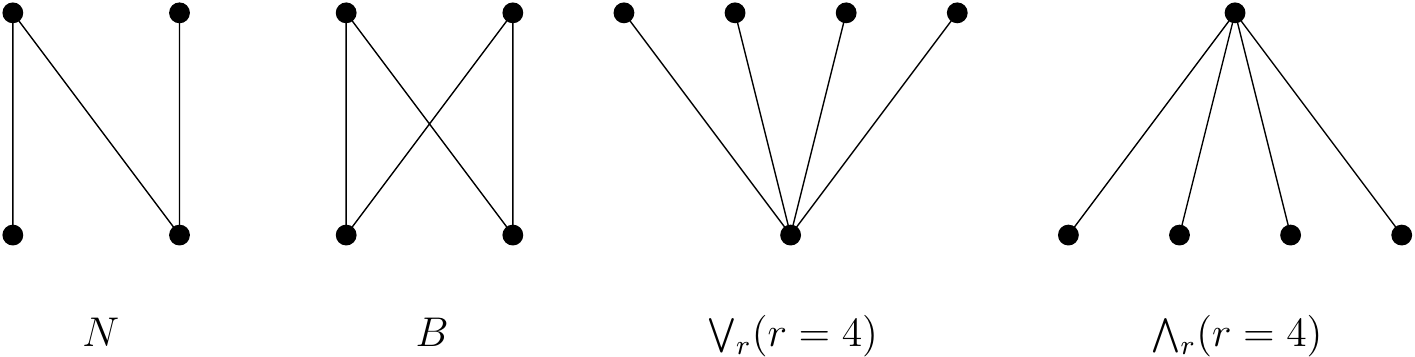}
\caption{The Hasse diagrams of the posets $N$, $B$, $\bigvee_r$ and $\bigwedge_r$ }
\label{fig:defs}
\end{figure}

\vskip 0.3truecm

A family that does not contain $P_2$ is called an antichain. A family $\cF$ that does not contain $P_k$ can be easily partitioned into $k-1$ antichains $\cF_1, \dots, \cF_{k-1}$ the following way: let $\cF_i$ be the set of minimal elements of $\cF\setminus \cup_{j=1}^{i-1} \cF_j$. We call this the \textit{canonical partition} of $\cF$. 

Our first theorem relies on some easy observations.

\begin{theorem}
\label{thm:easy}
\textbf{(a)} $La(n,\bigvee,P_2)=La(n,\bigwedge,P_2)={n\choose \lfloor n/2\rfloor}$.

\textbf{(b)} $La(n,\{\bigvee,\bigwedge\},P_2)={n-1 \choose \lfloor (n-1)/2\rfloor}$.

\textbf{(c)} $La(n,B,D_r)={{n \choose \lfloor n/2\rfloor}\choose r}$.

\textbf{(d)} $La(n,\bigvee, \bigwedge_r)=La(n,\bigwedge,\bigvee_r)={{n \choose \lfloor n/2\rfloor} \choose r}$
\end{theorem}

\vskip 0.2truecm
The proof of our next theorem uses the notion of profile vectors (ordinary and $l$-chain profile vectors). Here and throughout the paper $h(x)$ denotes the binary entropy function, i.e. $h(x)=-x\log_2x-(1-x)\log_2(1-x)$.

\begin{theorem}
\label{thm:profile}
\textbf{(a)} $La(n,P_3,\bigwedge_r)=La(n,P_3,\bigvee_r)=\binom{n}{i_r}\binom{\binom{i_r}{\lfloor i_r/2\rfloor}}{r}$ for some $i_r$ with $i_r=(1+o(1))\frac{2^r}{2^r+1}n$.

\textbf{(b)} $La(n,P_4,D_r)=\binom{n}{j_r}\binom{j_r}{i_r}\binom{\binom{j_r-i_r}{\lfloor (j_r-i_r)/2\rfloor}}{r}$ for some $i_r=(1+o(1))\frac{n}{2^r+2}$ and either $j_r=n-i_r$ or $j_r=n-i_r-1$.

\textbf{(c)} $2^{(c+o(1))n} \le La(n,P_3,N) \le o(2^{3n})$,

where $c=h(c_0)+3c_0h(c_0/(1-c_0))=2.9502...$ with $c_0$ being the real root of the equation $0=7x^3-10x^2+5x-1$.
\end{theorem}

Let us return for a moment to the original forbidden subposet problems. The main conjecture of the area was first published by Griggs and Lu in \cite{GL}. 

\begin{conjecture}
\label{conj:main}
For a poset $P$ let us denote by $e(P)$ the largest integer $m$ such that for any $n$, any family $\cF\subseteq 2^{[n]}$ consisting of $m$ \textit{consecutive} levels is $P$-free. Then \[\lim_n\frac{La(n,P,P_1)}{\binom{n}{\lceil n/2\rceil}}=e(P)
\]holds.
\end{conjecture}

In words, \cjref{main} states that to obtain an asymptotically largest $P$-free family $\cF\subseteq 2^{[n]}$ one has to consider as many middle levels of $2^{[n]}$ as possible without creating a copy of $P$. Again, we refer the interested Reader to the recent survey \cite{GLi} to see for which families of posets \cjref{main} has been verified.

\vskip 0.2truecm

However, already \tref{gp} shows that to make \cjref{main} valid in the more general context one has to remove at least the word consecutive. All parts of both \tref{easy} and \tref{profile} suggest that  a general conjecture stating that for any pair $P,Q$ of posets $La(n,P,Q)$ is asymptotically attained at a sequence of families consisting of full levels of $2^{[n]}$. But this is not the case at all!  There are pairs of posets for which all families consisting of full levels are very far from being optimal. Let us consider $La(n,D_2,P_3)$ the maximum number of 3-chains in diamond-free families. Every family that contains at least three full levels of $2^{[n]}$ contains a copy of $D_2$, while a family that is the union of at most two levels, does not contain any copy of $P_3$. Therefore, if $\cF$ is $D_2$-free and is the union of full levels, then $c(P_3,\cF)=0$, while there are $D_2$-free families with lots of copies of $P_3$. Note that $D_2$ is the smallest poset for which \cjref{main} has not been proved.
 
 \begin{theorem}\label{thm:diamond}
 For the generalized diamond posets  and integers $k>l$ the following holds:
 $$\binom{k-1}{l}La(n-k+1,P_3,P_2)\le La(n,D_k,D_l)\le \left(\binom{k+1}{2}-k\right)\binom{k-1}{l}La(n,P_3,P_2).$$
 \end{theorem}
 
 Note that \tref{diamond} implies $La(n,D_k,D_l)=\theta_{k,l}(La(n,P_3,P_2))$ for any fixed $k$ and $l$ and the exact value of $La(n,P_3,P_2)$ is given by \tref{gp}. So it is a natural question whether the limit $d_{k,l}=\lim_{\infty}\frac{La(n,D_k,D_l)}{La(n,P_3,P_2)}$ exists and if so, what its value is. In the simplest case $k=2,l=1$ the above inequalities and \tref{gp} imply $1/3\le d_{2,1}\le 1$.
 \vskip 0.2truecm
 So what can be saved from \cjref{main} in the more general context? Let $l(P)$ be the height of a poset $P$, i.e. the length of the longest chain in $P$. Clearly, if $\cF$ is the union of any $l(P)-1$ full levels, it must be $P$-free. On the other hand if  $\cF_n=\cup_{j=1}^{l(Q)}\binom{[n]}{i_j}$ is the union of $l(Q)$ full levels with $i_{j+1}-i_j\ge cn$ for some constant $c$ for all $j=1,2,\dots l(Q)-1$, then $\cF_n$ contains many copies of $Q.$
 Therefore we propose the following.

\begin{conjecture}\label{conj:genconj}
For any pair $P,Q$ of posets with $l(P)>l(Q)$ there exist a sequence of $P$-free families $\cF_n\subseteq 2^{[n]}$ all of which are unions of full levels such that
\[
La(n,P,Q)=(1+o(1))c(Q,\cF_n)
\]
holds.
\end{conjecture}

As we have already seen, this conjecture often holds even if $l(P)\le l(Q)$. We say that for a pair $P,Q$ of posets \cjref{genconj} \textit{strongly holds} if for large enough $n$ we have $La(n,P,Q)=c(Q,\cF_n)$ and \textit{almost holds} if   $La(n,P,Q)=O( n^kc(Q,\cF_n))$ for some $k$ that depends only on $P$ and $Q$. In both cases we also assume the family $\cF_n$ is $P$-free and is the union of full levels, but we do not assume anything about $l(P)$ and $l(Q)$. Parts \textbf{(a)} \textbf{(c)}, and \textbf{(d)} of \tref{easy} show that \cjref{genconj} strongly holds for those pairs of posets. 
\vskip 0.2truecm

In \tref{easy} and \tref{profile} we dealt with $La(n,P_{l(Q)+1},Q)$ for different posets $Q$. (In the case of $La(n,B,D_r)$ it is implicit, as the $B$-free property implies $P_4$-free property.) We knew the place of every element of  every copy of $Q$ in the canonical partition. In the following we deal with these kind of problems. We introduce the following binary operations of posets: for any pair $Q_1$, $Q_2$ of posets  we define $Q_1\otimes_r Q_2$ by adding an antichain of size $r$ between $Q_1$ and $Q_2$. More precisely, let us assume $Q_1$ consists of $q^1_1,\dots, q^1_a$ and $Q_2$ consists of $q^2_1,\dots,q^2_b$. Then $R=Q_1\otimes_r Q_2$ consists of $q^1_1,\dots, q^1_a,m_1,m_2,\dots, m_r,q^2_1,\dots q^2_b$. We have $q^1_i<_R q^1_j$ if and only if $q^1_i<_{Q_1} q^1_j$ and similarly $q^2_i<_R q^2_j$ if and only if $q^2_i<_{Q_2} q^2_j$. Also we have $q^1_i<_R m_k<_R q^2_j$ for every $i$,  $k$, and $j$. Finally, the $m_k$'s form an antichain. Note that $l(Q_1\otimes_r Q_2)=l(Q_1)+l(Q_2)+1$.
Let $Q\oplus r$ denote the poset $Q\otimes_r \textbf{0}$, where $\textbf{0}$ is the empty poset, i.e. $Q\oplus r$ is obtained from $Q$ by adding $r$ elements that form an antichain and that are all larger than all elements of $Q$. Similar operations of posets were considered first in the area of forbidden subposet problems by Burcsi and Nagy \cite{BN}.

We will obtain bounds on $La(n,P_{l(Q_1\otimes_r Q_2)+1},Q_1\otimes_rQ_2)$ involving bounds on $La(n,P_{l(Q_1)+1},Q_1)$ and $La(n,P_{l(Q_2)+1},Q_2)$. For this we will need the following auxiliary statement that can be of independent interest.

\begin{theorem}\label{thm:rtuples}
\textbf{(a)} For every $r\ge 3$ and antichain $\cA\subseteq 2^{[n]}$ the number $\gamma^r_{0,n}(\cA)$ of $r$-tuples $A_1,A_2\dots, A_r$ with $|\bigcap_{i=1}^rA_i|=0$ and $\bigcup_{i=1}^rA_i=[n]$ is at most $n^{2r}\gamma^r_{0,n}(\binom{[n]}{\lfloor n/2\rfloor})$. If $r=2$ and $n$ is even, then $\gamma^2_{0,n}(\cA)\le \gamma_{0,n}^2(\binom{[n]}{n/2})$, while if $r=2$ and $n$ is odd, then $\gamma^2_{0,n}(\cA)\le \binom{n-1}{\lfloor n/2\rfloor -1}$.

\textbf{(b)} For every $r$ there exists a sequence $l_n$ such that if $\cA\subseteq 2^{[n]}$ is an antichain, then the number $\beta^r_0(\cA)$ of $r$-tuples $A_1,A_2\dots, A_r$ with $|\bigcap_{i=1}^rA_i|=0$ is at most $n^{2r+1}\beta^r_0(\binom{[n]}{l_n})$.

\textbf{(c)} If $\cA\subseteq 2^{[n]}$ is an antichain, then $\beta^2_0(\cA)\le\frac{1}{2}\binom{n}{\lfloor n/3\rfloor}\binom{\lceil 2n/3\rceil}{\lfloor n/3\rfloor}$ and this is sharp as shown by $\binom{[n]}{\lfloor n/3\rfloor}$ if $n\equiv 0,1$ mod 3 and by $\binom{[n]}{\lceil n/3\rceil}$ if $n\equiv 2$ mod 3.
\end{theorem}

The reason for the strange indices is that we will prove a somewhat more general result in Section 4. The $r=2$ part of \tref{rtuples} \textbf{(a)} was proved by Bollob\'as \cite{B}.

\begin{theorem}\label{thm:ragasztas1} Let $Q_1,Q_2$ be two non-empty posets.

\textbf{(a)} If $r\ge 2$, then we have
\begin{multline*}
La(n,P_{l(Q_1)+l(Q_2)+1},Q_1\otimes_r Q_2)\le  \\
 n^{2r+2}\max_{0\le i< j\le n}\left\{\binom{n}{j}\binom{j}{i}\gamma^r_{0,j-i}\left(\binom{[j-i]}{\lfloor (j-i)/2\rfloor}\right)La(i,P_{l(Q_1)+1},Q_1)La(n-j,P_{l(Q_2)+1},Q_2)\right\}.
\end{multline*}
Furthermore, if $r\ge 3$ and \cjref{genconj} almost holds for the pairs $P_{l(Q_1)+1},Q_1$ and $P_{l(Q_2)+1}$, $Q_2$, then so it does for the pair $P_{l(Q_1\otimes_r Q_2)+1},Q_1\otimes_r Q_2$.

\textbf{(b)} If $r=1$, then we have
$$La(n,P_{l(Q_1)+l(Q_2)+1},Q_1\otimes_1 Q_2)\le \max_{0\le j\le n}\left\{\binom{n}{j}La(j,P_{l(Q_1)+1},Q_1)La(n-j,P_{l(Q_2)+1},Q_2)\right\}$$

Furthermore, if \cjref{genconj} strongly/almost holds for the pairs $P_{l(Q_1)+1},Q_1$ and $P_{l(Q_2)+1}$, $Q_2$, then so it does for the pair $P_{l(Q_1\otimes_1 Q_2)+1},Q_1\otimes_1 Q_2$.
\end{theorem}

\begin{theorem}
\label{thm:ragasztas2}
Let $Q$ be a non-empty poset.

\textbf{(a)} If $r\ge 2$ and $n\in \mathbb{N}$, then there exists an $i=i(r,n)$ such that
$$La(n,P_{l(Q)+2},Q\oplus r)\le \max_{0\le j\le n}\left\{\binom{n}{i}\beta_0^r\left(\binom{[n-i]}{j-i}\right)La(j,P_{l(Q)+1},Q)\right\}.$$
Furthermore, if \cjref{genconj} almost holds for the pair $P_{l(Q)+1},Q$, then so it does for the pair $P_{l(Q)+2},Q_1\oplus r $.

\textbf{(b)} If $r=1$, then we have
$$La(n,P_{l(Q)+2},Q\oplus 1)\le \max_{0\le j\le n}\left\{\binom{n}{j}La(j,P_{l(Q)+1},Q)\right\}$$
Furthermore, if \cjref{genconj} strongly/almost holds for the pair $P_{l(Q)+1},Q$, then so it does for the pair $P_{l(Q)+2},Q\oplus 1 $.
\end{theorem}

\begin{figure}[h]
\centering
\includegraphics{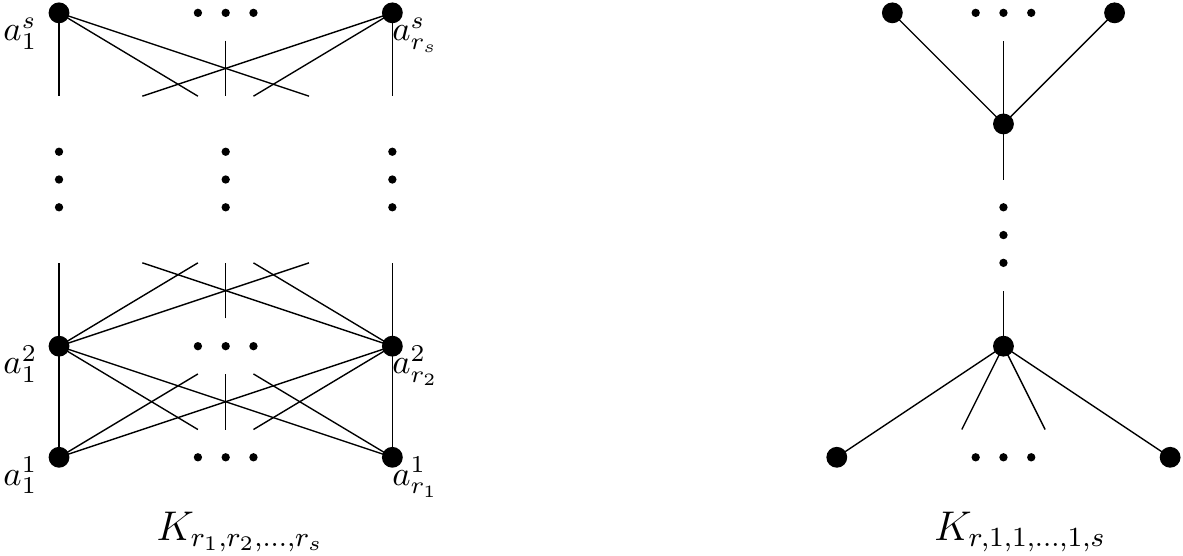}
\caption{The Hasse diagrams of multi-level posets.}
\label{fig:defs2}
\end{figure}

We can apply \tref{ragasztas1} and \tref{ragasztas2} to complete multi-level posets. Let $K_{r_1,r_2,\dots,r_s}$ denote the poset on $\sum_{i=1}^s r_i$ elements $a_1^1, a^1_2,\dots , a^1_{r_1},a^2_1,a^2_2,\dots, a^2_{r_2},\dots , a^s_1, a^s_2, \dots,a^s_{r_s}$ with $a^i_h < a^j_l$ if and only if $i < j$. Observe that $\bigvee_r=K_{1,r}$, $B=K_{2,2}$ and $D_r=K_{1,r,1}$.  

\begin{corollary}
\label{cor:genmult} For any complete multi-level poset $K_{r_1,r_2,\dots,r_s}$ \cjref{genconj} almost holds for the pair $P_{s+1},K_{r_1,r_2,\dots,r_s}$.
\end{corollary}

\begin{corollary}\label{cor:genmult1} \cjref{genconj} strongly holds for the pair $P_{s+1},K_{r_1,r_2,\dots,r_s}$ if  for every $i<s$ at least one of $r_i$ and $r_{i+1}$ is equal to $1$. 
\end{corollary}

\cref{genmult1} does not tell us anything about the set sizes in the family containing the most number of copies of $K_{r_1,r_2,\dots,r_s}$. The next theorem gives more insight for an even more special case.

\begin{theorem} \label{thm:multi1}
The value of $La(n,P_{l+3},K_{r,1,\dots,1,s})$ is attained for a family $\cF=\cup_{j=1}^{l+2}\binom{[n]}{i_j}$, where $i_1=\lfloor i_2/2\rfloor$, $i_{l+2}=\lfloor (n+i_{l+1})/2\rfloor$ and $i_3-i_2$, $i_4-i_3, \dots, i_{l+1}-i_l$ differ by at most 1.
\end{theorem}

The rest of the paper is organized as follows: we prove \tref{diamond} and \tref{easy} in Section 2. We explain the profile polytope method and prove \tref{profile} in Section 3. \tref{rtuples}, \tref{ragasztas1}, \tref{ragasztas2}, and their corollaries  are proved in Section 4, while Section 5 contains some concluding remarks.

\section{Proofs of \tref{diamond} and \tref{easy}}

\begin{proof}[Proof of \tref{diamond}]
We start by proving the upper bound. Let $\cF\subseteq 2^{[n]}$ be a $D_k$-free family. As for any poset $P$  the canonical partition of a $P$-free family can consist of at most $|P|-1$ subfamilies, we can assume that the canonical partition of $\cF$ is $\cup_{i=1}^{k+1}\cF_i$. In any copy of $D_l$ in $\cF$, the sets corresponding to the top and bottom element of $D_l$ come from $\cF_i$ and $\cF_j$ with $i-j\ge 2$. The number of such pairs of indices is $\binom{k+1}{2}-k$. Let us bound the number of copies of $D_l$ with top element from $\cF_i$ and bottom element from $\cF_j$. As $\cF_i\cup \cF_j$ is $P_3$-free, there are at most $La(n,P_3,P_2)$ many ways to choose the top and the bottom elements $F_B\subset F_T$. As $\cF$ is $D_k$-free there can be at most $k-1$ sets in $\cF$ lying between $F_B$ and $F_T$, so the number of copies of $D_l$ with $F_B,F_T$ being top and bottom is at most $\binom{k-1}{l}$. The upper bound on $La(n,D_k,D_l)$ follows.

For the lower bound we need a construction. Let $\cF_1\cup \cF_2  \subseteq 2^{[n-k+1]}$ be the canonical partition of the $P_3$-free family $\cF$ with $c(P_2, \cF)=La(n-k+1,P_3,P_2)$. For $j=3,4,\dots, k+1$ let $\cF_j=\{F\cup [n-k,n-k+j-1]:F\in \cF_2\}$. We claim that $\cG=\cup_{i=1}^{k+1}\cF_i$ is $D_k$-free with $c(D_l, \cF)\ge \binom{k-1}{l}La(n-k+1,P_3,P_2)$. Indeed, every set $G\in \cG$ is contained in a set $F_{k+1}\in \cF_{k+1}$ and contains a set $F_1\in\cF_1$, therefore if there was a copy of $D_k$, we could assume that its bottom element is from $\cF_1$ and its top element is from $\cF_{k+1}$. But any $F_{k+1}\in\cF_{k+1}$ contains exactly one element from each $\cF_i$ where $i=2,3,\dots,k$, so there is no space for a copy of $D_k$. On the other hand, for every pair $F_1\subset F_2$ in $\cF_1\cup \cF_2$ we can add $l$ sets from $\{F_2\cup [n-k,n-k+j-1]: j=3,4,\dots,k+1\}$ to form a copy of $D_l$. For each such pair we will obtain $\binom{k-1}{l}$ such copies.
\end{proof}




\begin{proof}[Proof of \tref{easy}] To prove \textbf{(a)}, by symmetry, it is enough to show $La(n,\bigvee,P_2)=\binom{n}{\lfloor  n/2\rfloor}$. Consider  any $\bigvee$-free family $\cF\subseteq 2^{[n]}$ and its canonical partition $\cF_1\cup \cF_2$. By the $\bigvee$-free property of $\cF$, elements of $\cF_1$ are contained in at most one copy of $P_2$. Also, as the $\bigvee$-free property implies the $P_3$-free property, every copy of $P_2$ in $\cF$ must contain a set from $\cF_1$. Sperner's theorem yields $c(P_2,\cF)\le |\cF_1|\le \binom{n}{\lfloor  n/2\rfloor}$.  On the other hand $\cF:=\{F\subseteq [n]: |F|=\lfloor n/2\rfloor\} \cup \{[n]\}$ is $\bigvee$-free and every $\lfloor n/2\rfloor$-element set forms a copy of $P_2$ with $[n]$.

We continue with proving \textbf{(b)}. We will need the following definition. For any family $\cF$, the \textit{comparability graph of $\cF$} has vertex set $\cF$ and two sets $F,F'\in \cF$ are joined by an edge if $F\subseteq F'$ or $F'\subseteq F$ holds. The connected components of the comparability graph of $\cF$ are said to be the \textit{components of $\cF$.} If a family $\cF$ is both $\bigvee$-free and $\bigwedge$-free, then its components are either isolated vertices or isolated edges in the comparability graph. Therefore $c(P_2,\cF)$ is the number of components that are isolated edges. It follows that $La(n,\{\bigvee,\bigwedge\},P_2)\le \frac{1}{2}La(n,\{\bigvee,\bigwedge\},P_1)={n-1 \choose \lfloor (n-1)/2\rfloor}$ where the result in the last equation was proved by Katona and Tarj\'an \cite{KT}. The construction (given also in \cite{KT}) $\cF:=\binom{[n-1]}{\lfloor (n-1)/2\rfloor}\cup \{\{n\}\cup F:F\in\binom{[n-1]}{\lfloor (n-1)/2\rfloor}\}$ shows that the above upper bound can be attained.

To prove \textbf{(c)}, let us consider a $B$-free family $\cF\subseteq 2^{[n]}$ and let $\cM=\{M \in \cF: \exists F',F''\in \cF \,\,\textrm{such that}\, F'\subsetneq M\subsetneq F''\}$. As $B$-free implies $P_4$-free, we obtain that $\cM$ is an antichain, thus $|\cM|\le \binom{n}{\lfloor n/2\rfloor}$, by \tref{sperner}. Moreover, if $M\in \cM$, then there do not exist two elements $F_1,F_2\in \cF$ with $M\subsetneq F_1,F_2$. Indeed, by the definition of $\cM$ there exists $F' \in \cF$ with $F'\subsetneq M$, and $F',M,F_1,F_2$ would form a copy of $B$. Similarly, for every $M\in \cM$ there exists exactly one element $F\in \cF$ with $F\subsetneq M$. Therefore a copy of $D_r$ contains $r$ elements of $M$, and they determine the remaining two elements, which implies $c(D_r, \cF)\le {|\cM| \choose r}\le {{n \choose \lfloor n/2\rfloor}\choose r}$.
The construction $\cF:=\{\emptyset,[n]\}\cup \binom{[n]}{\lfloor n/2\rfloor}$ shows that this upper bound can be attained.

To prove \textbf{(d)}, by symmetry, it is enough to show $La(n,\bigvee,\bigwedge_r)={{n \choose \lfloor n/2\rfloor} \choose r}$. If $\cF$ is $\bigvee$-free, then it is in particular $P_3$-free. Consider its canonical partition. Then a copy of $\bigwedge_r$ contains $r$ elements from $\cF_1$ and one from $\cF_2$. Moreover, an $r$-tuple from $\cF_1$ may form a copy of $\bigwedge_r$ with at most one element from $\cF_2$, otherwise there is a copy (actually $r$ copies) of $\bigvee$ in $\cF$. As $\cF_1$ is an antichain, by \tref{sperner}, the upper bound $La(n,\bigvee,\bigwedge_r)\le {{n \choose \lfloor n/2\rfloor} \choose r}$ follows and $\cF:=\{[n]\}\cup \binom{[n]}{\lfloor n/2\rfloor}$ shows that this can be attained.
\end{proof}

\section{The profile polytope method}
In this section we prove \tref{profile} after introducing the notions of profile vectors and profile polytopes. For a family $\cF \subseteq 2^{[n]}$ of sets, let $\alpha(\cF)=(\alpha_0,\alpha_1,\dots, \alpha_n)$ denote the \textit{profile vector} of $\cF$, where $\alpha_i=|\{F \in \cF: |F|=i\}|$. Many problems in extremal finite set theory ask for the largest size of a family in a class $\mathbb{A}\subseteq 2^{2^{[n]}}$. This question is equivalent to determining $\max_{\cF\in \mathbb{A}}\alpha(\cF)\cdot \mathbf{1}$, where $\mathbf{1}$ is the vector of length $n+1$ with all entries being 1, and $\cdot$ denotes the scalar product. 

More generally, consider a weight function $w: \{0,1, \dots, n\} \rightarrow \mathbb{R}$, and assume we want to maximize $w(\cF):=\sum_{F\in \cF}w(|F|)$. Then this is equivalent to maximizing $\alpha(\cF)\cdot \mathbf{w}$, where $\mathbf{w}=(w(0),w(1),\dots,w(n))$.
As $\mathbb{A}\subseteq 2^{2^{[n]}}$ holds, we have $\{\alpha(\cF):\cF\in\mathbb{A}\}\subseteq \mathbb{R}^{n+1}$ and therefore we can consider its convex hull $\mu(\mathbb{A})$ that we call the \textit{profile polytope} of $\mathbb{A}$. 
It is well known that any weight function with the above property is maximized by an \textit{extreme point} of $\mu(\mathbb{A})$ (a point that is not a convex combination of other points of $\mu(\mathbb{A})$), moreover if such a weight function is non-negative, then it is maximized by an \textit{essential extreme point}, i.e. an extreme point which is maximal with respect to the coordinate-wise ordering. First results concerning profile polytopes were obtained in \cite{K2, EFK, EFK2,EFK3,FK} and the not too recent monograph of Engel \cite{En} contains a chapter devoted to this topic.


Using this we can determine $La(n,P_3,P_2)$, and using induction with this as the base case one can determine $La(n,P_{k+1},P_k)$, but in other cases we will need a more powerful tool than ordinary profile vectors. The notion of \textit{$l$-chain profile vector} $\alpha_l(\cF)$ of a family $\cF\subseteq 2^{[n]}$ was introduced by Gerbner and Patk\'os \cite{GP} and denotes a vector of length $\binom{n+1}{l}$. The coordinates are indexed by $l$-tuples of $[0,n]$ and $\alpha_l(\cF)(i_1,i_2,\dots,i_l)$ is the number of chains $F_1\subsetneq F_2\subsetneq \dots \subsetneq F_l$ such that $F_j\in \cF$ and $|F_j|=i_j$ for all $1\le j\le l$.
For a set $\mathbb{A}\subseteq 2^{2^{[n]}}$ one can define the $l$-chain profile polytope $\mu_l(\mathbb{A})$, its extreme points and essential extreme points analogously to the above. Note that for $l=1$ we get back the definition of the original profile polytope.

Let $\mathbb{S}_{n,k}$ be the class of all $k$-Sperner families on $[n]$.

\begin{lemma}[Gerbner, Patk\'os, \cite{GP}]
The essential extreme points of $\mu_l(\mathbb{S}_{n,k})$ are the $l$-chain vectors of $k$-Sperner families that consist of the union of $k$ full levels.
\end{lemma}

Let us state the immediate consequence of the above lemma that we will use in our proofs in the remainder of this section.

\begin{corollary}
\label{cor:lchain}
Let $l\le k$ and $w: \binom{2^{[n]}}{l} \rightarrow \mathbb{R}^+$ be a weight function such that $w(\{F_1,F_2,\dots F_l\})$ depends only on $|F_1|,|F_2|,\dots,|F_l|$. Then the maximum of $$\sum_{F_1 \subsetneq F_2 \subsetneq \dots \subsetneq F_l, F_i \in \cF}w(\{F_1,F_2,\dots,F_l\})$$ over all families $\cF \in \mathbb{S}_{n,k}$ is attained at some family that consists of $k$ full levels.
\end{corollary}

\begin{proof}[Proof of \tref{profile}.]
To prove \textbf{(a)} we show $La(n,P_3,\bigwedge_r)=\binom{n}{i_r}\binom{\binom{i_r}{\lfloor i_r/2\rfloor}}{r}$ as the other statement follows by symmetry. Let us consider the canonical partition of a $P_3$-free family $\cF$. Note that a copy of $\bigwedge_r$ contains exactly one element $F$ from $\cF_2$ and $r$ elements $F_1,F_2,\dots, F_r \in \cF_1$ with $F_i\subsetneq F$ for all $1,2,\dots,r$. Let us consider a set $F \in \cF_2$. The sets of $\cF_1$ contained in $F$ form an antichain, thus by \tref{sperner}, their number is at most $\binom{|F|}{\lfloor |F|/2\rfloor}$. Therefore, the number of copies of $\bigwedge_r$ that contain $F$ is at most $\binom{\binom{|F|}{\lfloor |F|/2\rfloor}}{r}$ and we obtain
\[
c(\bigwedge_r,\cF)\le \sum_{F\in \cF_2}\binom{\binom{|F|}{\lfloor |F|/2\rfloor}}{r} \le \max_{\cA \in \mathbb{S}_{n,1}}\sum_{A \in \cA}\binom{\binom{|A|}{\lfloor |A|/2\rfloor}}{r}.
\]
Therefore if we set $w(i):={{i\choose \lfloor i/2\rfloor} \choose r }$, then we can apply \cref{lchain} with $l=k=1$ to obtain $c(\bigwedge_r,\cF) \le \max_{0 \le i \le n} {n \choose i}w(i)$. On the other hand, the families $\cF(i)=\binom{[n]}{i}\cup \binom{[n]}{\lfloor i/2\rfloor}$ are $P_3$-free and $c(\bigwedge_r,\cF(i))=\binom{n}{i}w(i)$ showing $La(n,P_3,\bigwedge_r)=\max_{0 \le i \le n} {n \choose i}w(i)$.
To obtain the value of $i_r$ we need to maximize $f(i):=\binom{n}{i}w(i)$. Considering
\[
\frac{f(i)}{f(i+1)}=\frac{i+1}{n-i}\cdot \frac{\prod_{j=0}^{r-1}\left(\binom{i}{\lfloor i/2\rfloor}-j\right)}{\prod_{j=0}^{r-1}\left(\binom{i+1}{\lfloor(i+1)/2\rfloor}-j\right)}=(1+o(1))\frac{i+1}{2^r(n-i)}
\]
when $i$ tends to infinity with $n$. For constant values of $i$, the ratio $f(i)/f(i+1)$ is easily seen to be smaller than 1 (if $n$ is big enough), therefore the maximum of $f(i)$ is attained at $i_r=(1+o(1))\frac{2^r}{2^r+1}n$ as stated.

To prove \textbf{(b)} we consider the canonical partition of a $P_4$-free family $\cF\subseteq 2^{[n]}$. Any copy of $D_r$ in $\cF$ must contain one set from $\cF_1,\cF_3$ each and $r$ sets from $\cF_2$. For any $F_1\in \cF_1,F_3\in \cF_3$ with $F_1 \subset F_3$, the number of copies of $D_r$ containing $F_1$ and $F_3$ is $\binom{m}{r}$, where $m=|\cM_{F_1,F_3}|$ with $\cM_{F_1,F_3}=\{F\in \cF_2: F_1 \subset F\subset F_3\}$. As $\cM'_{F_1,F_3}=\{M\setminus F_1:M\in \cM_{F_1,F_3}\}$ is on antichain in $F_3\setminus F_1$, we have $m\le \binom{|F_3|-|F_1|}{\lfloor (|F_3|-|F_1|)/2\rfloor}$. Therefore, we obtain
$$c(D_r,\cF)\le\sum_{F_1\in \cF_1,F_3\in \cF_3, F_1\subset F_3}\binom{\binom{|F_3|-|F_1|}{\lfloor (|F_3|-|F_1|)/2\rfloor}}{r}\le \max_{0\le i<j\le n}\binom{n}{j}\binom{j}{i}\binom{\binom{j-i}{\lfloor (j-i)/2\rfloor}}{r}$$
where to obtain the last inequality we applied \cref{lchain} with $l=k=2$ and $w(i,j)=\binom{\binom{j-i}{\lfloor (j-i)/2\rfloor}}{r}$. Observe that if $i_r$ and $j_r$ are the values for which this maximum is taken, then for the family $\cF=\binom{[n]}{i_r}\cup \binom{[n]}{j_r}\cup \binom{[n]}{\lfloor (i_r+j_r)/2\rfloor}$ we have $c(D_r,\cF )=\binom{n}{j_r}\binom{j_r}{i_r}\binom{\binom{j_r-i_r}{\lfloor (j_r-i_r)/2\rfloor}}{r}$. 

To obtain the value of $i_r$ and $j_r$ let us fix $x=j-i$ first. Note that $\binom{n}{j}\binom{j}{i}=\binom{n}{x}\binom{n-x}{i}$, so we have
\[
\frac{\binom{n}{i+1+x}\binom{i+1+x}{i+1}\binom{\binom{x}{\lfloor x/2\rfloor}}{r}}{\binom{n}{i+x}\binom{i+x}{i}\binom{\binom{x}{\lfloor x/2\rfloor}}{r}}=\frac{\binom{n}{x}\binom{n-x}{i+1}\binom{\binom{x}{\lfloor x/2\rfloor}}{r}}{\binom{n}{x}\binom{n-x}{i}\binom{\binom{x}{\lfloor x/2\rfloor}}{r}}=\frac{n-x-i}{i+1},
\]
which implies that $i_r+j_r=n$ or $i_r+j_r=n-1$ holds. 

Let $g(i)=\binom{n}{i}\binom{n-i}{i}\binom{\binom{n-i}{\lfloor (n-i)/2\rfloor}}{r}$, then $$\frac{g(i+1)}{g(i)}=(1+o(1))\frac{(n-2i-1)(n-2i-2)}{(i+1)^22^{2r}}.$$

This implies the maximum of $g(i)$ is attained at $i=(1+o(1))\frac{1}{\cdot 2^r+2}$ .

To prove \textbf{(c)} we again consider the canonical partition of a $P_3$-free family $\cF\subseteq 2^{[n]}$. A copy of $N$ in $\cF$ must contain two sets from $\cF_1$ and two from $\cF_2$. Let $a,b,c,d$ be the four elements of $N$ with $a \le c$ and $b \le c,d$. For every copy of $N$ in $\cF$ there is a bijection $\phi$ from $N$ to that copy. Then we count the copies of $N$ in $\cF$ according to the images $\phi(b),\phi(c)$. Clearly, they form a 2-chain in $\cF$, the possible images of $d$ form an antichain among those sets of $\cF_2$ that contain $\phi(b)$ and the possible images of $a$ form an antichain among those sets of $\cF_1$ that are contained by $\phi(c)$. Therefore, we obtain
\[
c(N,\cF)\le \sum_{F_1,F_2\in \cF,F_1\subset F_2}\binom{n-|F_1|}{\lfloor \frac{n-|F_1|}{2}\rfloor}\binom{|F_2|}{\lfloor \frac{|F_2|}{2}\rfloor}\le \max_{0\le i<j\le n}\binom{n}{j}\binom{j}{i}\binom{n-i}{\lfloor \frac{n-i}{2}\rfloor}\binom{j}{\lfloor \frac{j}{2}\rfloor},
\]
where to obtain the last inequality we applied \cref{lchain} with $l=k=2$ and $w(i,j)=\binom{n-i}{\lfloor \frac{n-i}{2}\rfloor}\binom{j}{\lfloor \frac{j}{2}\rfloor}$. We have $\binom{n}{j}\binom{j}{i}=\binom{n}{j-i}\binom{n-j+i}{i}$, thus we get $$c(N,\cF)\le\max_{0\le i<j\le n}\binom{n}{j-i}\binom{n-j+i}{i}\binom{n-i}{\lfloor \frac{n-i}{2}\rfloor}\binom{j}{\lfloor \frac{j}{2}\rfloor}=\max_{0\le i<j\le n}o(2^{n+n-j+i+n-i+j})=o(2^{3n}).$$

Note that $j=\lceil 3n/4\rceil $ and $i=\lceil n/4\rceil$ show the exponent cannot be improved with this method.


To obtain the lower bound consider the $P_3$-free families $\cF_{i,j}=\binom{[n]}{i}\cup \binom{[n]}{j}$ with $0 \le i<j\le n$. Observe that we have $c(N,\cF_{i,j})\ge \frac{1}{4}\binom{n}{j}\binom{j}{i}^2\binom{n-i}{j-i}=:g(i,j)$ (the $1/4$-factor is due to the fact that copies of $B$ are counted 4 times as copies of $N$). To maximize $g(i,j)$ we first fix $j-i$ and consider $$\frac{g(i+1,j+1)}{g(i,j)}=\frac{(n-j)^2(j+1)^2}{(i+1)^2(j+1)(n-i)}.$$ It is easy to see that this fraction becomes smaller than 1 when $i$ is roughly $n-j+1$. Thus the maximum of $g(i,j)$ is asymptotically achieved when $i=n-j$.  

Similarly, we have
$$\frac{g(n-j-1,j+1)}{g(n-j,j)}=\frac{(n-j)(j+1)^3(n-j)^3}{(j+1)(2j-n+2)^3(2j-n+1)^3}.$$
Therefore, if we write $j=(c+o(1))n$, we obtain that $g$ is maximized when $\frac{c^2(1-c)^4}{(2c-1)^6}=1$. After taking the square root of the expression on the left hand side, this is  equivalent to that $0=7c^3-10c^2+5c-1$ holds. The solution of this equation is $c_0=0.69922..$. As $g(n-j,j)=\frac{1}{4}\binom{n}{j}\binom{j}{n-j}^3=\Omega(2^{h(j/n)+3\frac{j}{n}h((n-j)/j)}/n^2)$, the lower bound follows by plugging in $j=0.69922n$.
\end{proof}

\section{The $\otimes_r$ operation and copies of complete multi-level posets
}
In this section we prove results concerning the binary operation $Q_1\otimes_r Q_2$. We introduce two types of profile vectors.

For a family $\cF \subseteq 2^{[n]}$ of sets, let $\beta^r(\cF)=(\beta^r_0,\beta^r_1,\dots, \beta^r_{n-1})$ denote the \textit{$r$-intersection profile vector} of $\cF$, where $\beta^r_i=\beta^r_i(\cF)=|\{\{F_1,F_2,\dots,F_r\}: F_j \in \cF, \textrm{these are $r$ different sets and}\ $ $|\bigcap_{j=1}^r F_j|=i\}|$.

For a family $\cF \subseteq 2^{[n]}$ of sets, let $\gamma^r(\cF)=(\gamma_{0,1}^r,\gamma^r_{0,2},\dots, \gamma_{0,n}^r, \gamma_{1,2}^r, \dots, \gamma_{n-1,n}^r)$ denote the $r$-\textit{intersection-union profile vector} of $\cF$, where $\gamma_{i,j}^r=\gamma_{i,j}^r(\cF)=|\{\{F_1,\dots,F_r\}: F_1,\dots,F_r \in \cF, \textrm{these are $r$ different sets}, |F_1\cap \dots\cap F_r|=i, |F_1,\cup\dots\cup F_r|=j\}|$. Note that if $\cA\subseteq 2^{[n]}$ is an antichain, then $\gamma^r_{i,j}(\cA)>0$ implies $j-i\ge 2$, therefore the number of non-zero coordinates in $\gamma^r(\cA)$ is at most $\binom{n+1}{2}-n=\binom{n}{2}\le n^2$. 

Let us illustrate with two examples why these profile vectors can be useful in counting copies of different posets. Let $\cF$ be a $P_3$-free and $\cG$ be a $P_4$-free family. We will estimate $c(K_{p,r},\cF)$ and $c(K_{p,r,s},\cG)$. If we consider the canonical partitions of $\cF=\cF_1\cup \cF_2$ and $\cG=\cG_1\cup \cG_2\cup \cG_3$, then a copy of $K_{p,r}$ in $\cF$ contains $p$ sets from $\cF_1$ and $r$ sets from $\cF_2$. If we fix $F_1,\dots, F_r\in \cF_2$, then the $p$ "bottom" sets of the copies of $K_{p,r}$ in $\cF$ containing $F_1,\dots, F_r$ form an antichain in $\{F\in \cF_1: F\subseteq \bigcap_{j=1}^rF_j\}$. Therefore, by \tref{sperner}, the number of these copies is at most $\binom{\binom{|\bigcap_{j=1}^rF_j|}{\lfloor |\bigcap_{j=1}^rF_j|/2\rfloor}}{p}$, so summing up for all possible $r$-tuples of $\cF_2$ we obtain $c(K_{p,r},\cF)\le \beta^r(\cF)\cdot \mathbf{w}_p$ and consequently
\[
La(n,P_3,K_{p,r}) \le \max\{\beta^r(\cA)\cdot \mathbf{w}_p:\cA \subseteq 2^{[n]} ~\text{is an antichain}\},
\]
where the $j$th entry of the vector $\mathbf{w}_p$ is $\binom{\binom{j}{\lfloor j/2\rfloor}}{p}$. 

Similarly, if we consider the canonical partition of $\cG=\cG_1\cup \cG_2\cup \cG_3$, then a copy of $K_{p,r,s}$ in $\cG$ contains $p$ sets from $\cG_1$, $r$ sets from $\cG_2$ and $s$ sets from $\cG_3$. If we fix $G_1,\dots, G_r\in \cG_2$, then the bottom $p$ and top $s$ sets of copies of $K_{p,r,s}$ containing $G_1,\dots, G_r$ form antichains in $\{G\in \cG_1: G\subseteq \bigcap_{j=1}^rG_j\}$ and $\{G\in \cG_3:G\supseteq \bigcup_{j=1}^rG_j\}$. Therefore, using again \tref{sperner}, we obtain $c(K_{p,r,s},\cG)\le \gamma^r(\cG_2)\cdot \mathbf{w}_{p,s}$ and consequently
\[
La(n,P_4,K_{p,r,s}) \le \max\{\gamma^r(\cA)\cdot \mathbf{w}_{p,s}:\cA \subseteq 2^{[n]} ~\text{is an antichain}\},
\]
where the $(i,j)$th entry of the vector $\mathbf{w}_{p,s}$ is $\binom{\binom{i}{\lfloor i/2\rfloor}}{p}\binom{\binom{n-j}{\lfloor (n-j)/2\rfloor}}{s}$.
Therefore determining the convex hull (and more importantly its extreme points) of the $r$-intersection profile vectors and $r$-intersection-union profile vectors of antichains would yield upper bounds on $La$-functions of complete multi-level posets.

We are not able to determine these convex hulls, we will only obtain upper bounds on the coordinates of these profile vectors. Note that \tref{rtuples} is about special coordinates, so the next two theorems imply that result.


\begin{theorem}\label{thm:gamma} 
\textbf{(a)} If $\cF\subseteq 2^{[n]}$ is an antichain and $j-i$ is even, then $\gamma^2_{i,j}(\cF)\le \gamma^2_{i,j}(\binom{[n]}{(i+j)/2})=\frac{1}{2}\binom{n}{j}\binom{j}{i}\binom{j-i}{(j-i)/2}$. If $j-i$ is odd, then $\gamma^2_{i,j}(\cF)\le \binom{n}{j}\binom{j}{i}\binom{j-i-1}{(\lfloor j-i)/2\rfloor -1}$.

\textbf{(b)} If $\cF$ is an antichain and $r\ge 3$, then $\gamma^r_{i,j}(\cF)\le n^{2r}\gamma^r_{i,j}(\binom{[n]}{\lfloor (i+j)/2\rfloor})$.

\end{theorem}

\noindent During the proof we will use several times that the number of pairs $A\subset B\subset [n]$ with $|A|=a,|B|=b$ is $\binom{n}{b}\binom{b}{a}=\binom{n}{b-a}\binom{n-(b-a)}{a}$. The first calculation is obvious, for the second pick first $B\setminus A$ from $[n]$ and then $A$ from $[n]\setminus (B\setminus A)$.

\begin{proof} To see \textbf{(a)}, we first consider the special case $i=0,j=n$. Observe that $\gamma^2_{0,n}(\cF)$ is the number of complement pairs in $\cF$. In an antichain, by \tref{sperner}, this is at most $|\cF|/2\le \binom{n}{\lfloor n/2\rfloor}/2$. If $n$ is even, then this is achieved when $\cF=\binom{[n]}{n/2}$, while the case of odd $n$ was solved by Bollob\'as \cite{B}, who showed that the number of such pairs is at most $\binom{n-1}{\lfloor n/2\rfloor -1}$ and this is sharp as shown by $\{F\in \binom{[n]}{\lfloor n/2\rfloor}:1\in F\}\cup\{F\in \binom{[n]}{\lceil n/2\rceil}:1\notin F\}$.

To see the general statement observe that for pair $I\subset J$ writing $\cF_{I,J}=\{F\in \cF: I\subseteq F\subseteq J\}$ we have $\gamma_{i,j}^2(\cF)=\sum_{I\in \binom{[n]}{i},J\in \binom{[n]}{j}}\gamma^2_{0,j-i}(\cF_{I,J})$. Therefore if $j-i$ is even we obtain $\gamma^2_{i,j}(\cF)\le \binom{n}{j}\binom{j}{i}\gamma_{0,j-i}^2(\binom{[j-i]}{(j-i)/2})=\gamma^2_{i,j}(\binom{[n]}{(j+i)/2})$, while if $j-i$ is odd, we obtain $\gamma^2_{i,j}(\cF)\le \binom{n}{j}\binom{j}{i}\binom{j-i-1}{(\lfloor j-i)/2\rfloor -1}$.
\vskip 0.2truecm
To show \textbf{(b)} it is enough to prove the statement fo $i=0,j=n$. Indeed, $\gamma^r_{i,j}(\cF)=\sum_{I,J}\gamma^r_{0,j-i}(\cF_{I,J})\le \binom{n}{j}\binom{j}{i}\gamma^r_{0,j-i}(\binom{[j-i]}{\lfloor (j-i)/2\rfloor}=\gamma^r_{i,j}\binom{[n]}{\lfloor (j+i)/2\rfloor}$.

We proceed by induction on $r$. We postpone the proof of the base case $r=3$, and assume the statement holds for $r-1$ and any $i<j$. Let us consider $r-1$ sets $F_1,\dots,F_{r-1}$ of $\cF$ and examine which sets can be added to them as $F_r$ to get empty intersection and $[n]$ as the union. Let $\cF'$ be the family of those sets. Let $A=\cap_{l=1}^{r-1} F_l$ and $B=\cup_{l=1}^{r-1} F_l$ with $a=|A|$ and $b=|B|$. Then members of $\cF'$ contain the complement of $B$ and do not intersect $A$, and $\cF'$ is Sperner. If we remove $\overline{B}$ from them, the resulting family is Sperner on an underlying set of size $b-a$, thus have cardinality at most $\binom{b-a}{\lfloor (b-a)/2\rfloor}=:w(a,b)$. Note that we count every $r$-tuple $F_1,\dots,F_r$ exactly $r$ times. It implies 
$$r\gamma^r_{0,n}(\cF)\le \sum_{a<b}\gamma^{r-1}_{a,b}(\cF)w(a,b)=\gamma^{r-1}(\cF)\cdot \mathbf{w}\le n^2\max_{a<b}\gamma^{r-1}_{a,b}(\cF)w(a,b)$$. 

By induction this is at most $n^2n^{2r-2}\max_{a<b}\gamma^{r-1}_{a,b}(\binom{[n]}{\lfloor (a+b)/2\rfloor})w(a,b)$. Let 
\begin{equation*}
\begin{split}
f(a,b) = & \gamma^{r-1}_{a,b}\left(\binom{[n]}{\lfloor (a+b)/2\rfloor}\right)w(a,b)\\
= & \binom{n}{b-a}\binom{n-(b-a)}{a}\gamma_{0,b-a}^{r-1}\left(\binom{ [b-a]}{\lfloor (a+b)/2\rfloor-a}\right)\binom{b-a}{\lfloor (b-a)/2\rfloor}.
\end{split}
\end{equation*}
If we fix $b-a$  and consider $\frac{f(a,b)}{f(a+1,b+1)}=\frac{a+1}{n-(b-a)-a}$, we can see that the maximum is taken when $b+a=n$ or $b+a=n-1$ depending on the parity of $b-a$ and $n$. 

Let $a^*,b^*$ be the values for which the above maximum is taken. Note that for any $a^*<p<b^*$ we have $r\gamma^r_{0,n}(\binom{[n]}{p})\ge \binom{n}{b^*}\binom{b^*}{a^*}\gamma_{0,b^*-a^*}^{r-1}(\binom{ [b^*-a^*]}{p-a^*})\binom{b^*-a^*}{p-n+b^*}$, by counting only those $r$-tuples where the first $r-1$ sets have intersection of size $a^*$ and union of size $b^*$. (This way we count those $r$-tuples at most $r$ times). This is exactly $f(a^*,b^*)$ if $p=\lfloor n/2\rfloor=\lfloor (a^*+b^*)/2\rfloor$, so we obtained $r\gamma^r_{0,n}(\cF)\le n^{2r}r\gamma^r_{0,n}(\binom{[n]}{\lfloor n/2\rfloor})$ as required.

For $r=3$ we similarly consider two members of $\cF$ and examine which sets can be added to them to get empty intersection and $[n]$ as the union. This leads to 
$$3\gamma_{0,n}^3(\cF)\le n^2\max_{a<b}\gamma^2_{a,b}(\cF)w(a,b).$$

Note that if the maximum is taken at $a'$ and $b'$ with $b'-a'$ even, then 
part \textbf{(a)} of the theorem gives $\gamma^2_{a',b'}(\cF)\le \gamma^2_{a',b'}(\binom{[n]}{(b'+a')/2})$ so $3\gamma^3_{0,n}(\cF)\le n^2 \gamma^2_{a',b'}(\binom{[n]}{(b'+a')/2})w(a',b')$. This essentially lets us use $r=2$ as the base case of induction, and finish the proof of this case similarly to the induction step above.

Let us choose $a^*,b^*$ that maximizes this upper bound with $b^*-a^*=b'-a'$. Similarly to the computation about $f(a,b)$ , we have $a^*+b^*=n$ or $n-1$ depending on the parity of $n$. Then we obtain 
$3\gamma^3_{0,n}(\cF)\le n^2 \gamma^2_{a^*,b^*}(\binom{[n]}{(b^*+a^*)/2})w(a^*,b^*)\le n^2 \binom{n}{b^*}\binom{b^*}{a^*}\gamma^2_{0,b^*-a^*}(\binom{[b^*-a^*]}{(b^*+a^*)/2})w(a^*,b^*)$. The lower bound on $3\gamma^3_{0,n}(\binom{[n]}{\lfloor n/2\rfloor})$ is $\binom{n}{b^*}\binom{b^*}{a^*}\gamma_{0,b^*-a^*}^2(\binom{ [b^*-a^*]}{p-a^*})\binom{b^*-a^*}{p-a^*}$
as in the inductive step.

However, if $b'-a'$ is odd, then $\gamma^2_{a',b'}(\binom{[n]}{\lfloor (a'+b')/2\rfloor})=0$. But we know by part \textbf{(a)}
$$\gamma^2_{a',b'}(\cF)w(a',b')\le \binom{b'-a'-1}{(b'-a'-1)/2-1}\binom{n}{b'}\binom{b'}{a'}w(a',b').$$
Similarly to the previous cases, if $b'-a'$ is fixed, then the maximum of the right hand side is taken for some $a^*,b^*$ with $b^*-a^*=b'-a'$ and $a'+b'=n$ if $n$ is odd, and  $a'+b'=n-1$ or $a'+b'=n+1$ if $n$ is even. Thus we can assume $\lfloor n/2 \rfloor= (a^*+b^*-1)/2$.
On the other hand, since $b^*-a^*$ is odd, we have $$3\gamma_{0,n}^3\left(\binom{[n]}{ (a^*+b^*-1)/2}\right)\ge \binom{n}{b^*-1}\binom{b^*-1}{a^*}\frac{1}{2}\binom{b^*-a^*-1}{(b^*-a^*-1)/2}\binom{b^*-1-a^*}{\frac{a^*+b^*-1}{2}-n+b^*-1},$$ 
by counting only those triples where two of the sets have intersection of size $a^*$ and union of size $b^*-1$. We can pick first the $(b^*-1)$-set $B$ and the $a^*$-set $A$ in $\binom{n}{b^*-1}\binom{b^*-1}{a^*}$ ways, then among $\{G \in \binom{[n]}{\lfloor (a^*+b^*)/2\rfloor}:A\subset G\subset B\}$ we can pick a pair $G_1,G_2$ with $G_1\cap G_2=A$, $G_1\cup G_2=B$ in $\binom{b^*-a^*-1}{(b^*-a^*-1)/2}/2$ ways and then the third set contains the complement of $B$ and does not intersect $A$. Using that $\binom{b^*-a^*-1}{\frac{a^*+b^*-1}{2}-n+b^*-1}\ge\binom{b^*-a^*-1}{(b^*-a^*-1)/2-1}$
this implies 
$$
3\gamma_{0,n}^3(\cF)\le 3n^2\gamma_{0,n}^3\left(\binom{[n]}{(a^*+b^*-1)/2}\right)\frac{n-b^*+1}{(b^*-a^*-1)/2}\le 3n^3\gamma_{0,n}^3\left(\binom{[n]}{\lfloor n/2\rfloor}\right),$$
as $b^*\ge n/2$.
\end{proof}

\begin{theorem}\label{thm:intuni}
\textbf{(a)} For any antichain $\cF\subseteq 2^{[n]}$ we have $\beta^2_i(\cF)\le \beta^2_i(\binom{[n]}{j(i)})$, where $j(i)=i+\lfloor (n-i)/3\rfloor$ if $n-i\equiv 0,1$ mod 3 and $j(i)=i+\lceil (n-i)/3\rceil$ if $n-i\equiv 2$ mod 3.

\textbf{(b)} For every $r\ge 3$ and $i\le n$ there exists  $j(r,i,n)$ such that $\beta^r_{i}(\cF)\le n^{2r+1}\beta^r_{i}(\binom{[n]}{j(r,i,n)})$ holds for any antichain $\cF\subseteq 2^{[n]}$.
\end{theorem}

\begin{proof}
First we prove \textbf{(a)} for the special case $i=0$. Let $\cF \subseteq 2^{[n]}$ be an antichain, and let $\overline{\cF}=\{\overline{F}:F\in \cF\}$, where $\overline{F}=[n]\setminus F$. As $\cF$ is an antichain, so is $\overline{\cF}$, and thus $\cF\cup \overline{\cF}$ is $P_3$-free. Note that for every pair $F_1,F_2 \in \cF$ with $|F_1\cap F_2|=0$ and $F_1\cup F_2\neq [n]$, we have two 2-chains $F_1 \subsetneq \overline{F_2}$ and $F_2 \subsetneq \overline{F_1}$. Also, every 2-chain in $\cF\cup \overline{\cF}$ comes from a pair $F_1,F_2 \in \cF$ with $|F_1\cap F_2|=0$ and $F_1\cup F_2\neq [n]$. 

Therefore, if we take the canonical partition of $\cF\cup \overline{\cF}$ into $\cF_1\cup \cF_2$ and introduce the weight function
$w(F)=\frac{1}{2}\binom{|F|}{\lceil |F|/2\rceil}$ if $F\in \cF_2,\overline{F}\notin \cF_2$ and $w(F)=1/2$ if $F\in \cF_2, \overline{F}\in \cF_2$, then the number of disjoint pairs in $\cF$ equals $\sum_{F\in F_2}w(F)$. This weight function does not depend only on the size of $F$, but $w'(f)=\frac{1}{2}\binom{|F|}{\lceil |F|/2\rceil}$ does and obviously $w(F)\le w'(F)$ holds for all $F$'s. As proved by Katona in Theorem 3 of \cite{K} this weight function is maximized when $\cF_2=\binom{[n]}{\lceil 2n/3\rceil}$. As $\cF_2$ does not contain complement pairs, it also maximizes $w$.

To see the general statement of \textbf{(a)}, we can apply the special case to any $I\subseteq [n]$ and $\cF_I=\{F\setminus I: I\subseteq F\in\cF \}$. We obtain $$\beta^2_i(\cF)=\sum_{I\in \binom{[n]}{i}}\beta^2_0(\cF_I)\le \binom{n}{i}\beta^2_0\left(\binom{[n-i]}{j(i)-i}\right)=\beta^2_i\left(\binom{[n]}{j(i)}\right).$$

To see \textbf{(b)}, let $\cF\subseteq 2^{[n]}$ be antichain. Observe
\begin{equation*}
\begin{split}
\beta^r_i(\cF)=\sum_{j=i+1}^n\gamma^r_{i,j}(\cF)\le & n\max_{j: i+1\le j\le n}\gamma^r_{i,j}(\cF)\\
\le & n^{2r+1}\max_{j: i+1\le j\le n}\gamma^r_{i,j}\left(\binom{[n]}{\lfloor (i+j)/2\rfloor}\right)\le n^{2r+1}\beta^r_i\left(\binom{[n]}{j(r,i,n)}\right),
\end{split}
\end{equation*}
where $j(r,i,n)=\lfloor (i+j^*)/2\rfloor$ with $j^*$ being the value of $j$ that maximizes $\gamma^r_{i,j}(\binom{[n]}{\lfloor (i+j)/2\rfloor})$. The penultimate inequality follows from \tref{gamma}.
\end{proof}

\begin{proof}[Proof of \tref{ragasztas1}]
Let $Q_1,Q_2$ be non-empty posets and let us consider the canonical partition of a $P_{l(Q_1\otimes_r Q_2)+1}$-free family $\cF\subseteq 2^{[n]}$.  Then in any copy of $Q_1\otimes_r Q_2$ in $\cF$, if $F_1,\dots,F_r$ correspond to the $r$ middle elements forming an antichain, we must have $F_1,\dots, F_r \in \cF_{l(Q_1)+1}$. Also, if a copy of $Q_1\otimes_r Q_2$ contains $F_1,\dots, F_r$, then the sets corresponding to the $Q_1$ part of $Q_1\otimes_r Q_2$ must be contained in $\bigcap_{l=1}^r F_l$, while the sets corresponding to the $Q_2$ part of $Q_1\otimes_r Q_2$ must contain $\bigcup_{l=1}^rF_l$. Therefore the number of copies of $Q_1\otimes_r Q_2$ in $\cF$ that contain $F_1,\dots, F_r$ is at most $La(|\bigcap_{l=1}^r F_l|,P_{l(Q_1)+1},Q_1)\cdot La(n-|\bigcup_{i=1}^r F_l|,P_{l(Q_2)+1},Q_2)$. We obtained that the total number of copies of $Q_1\otimes_r Q_2$ in $\cF$ is at most
\begin{equation}\label{lala}
\sum_{F_1,\dots,F_r\in \cF_{l(Q_1)+1}}La(|\bigcap_{l=1}^r F_l |,P_{l(Q_1)},Q_1)\cdot La(n-|\bigcup_{l=1}^r F_l|,P_{l(Q_2)+1},Q_2).
\end{equation}
If $r\ge 2$, then grouping the summands in (\ref{lala}) according to the pair $(|\bigcap_{l=1}^r F_l|,|\bigcup_{l=1}^r F_l|)$ we obtain 
$$
La(n,P_{l(Q_1\otimes_r Q_2)+1},Q_1\otimes_r Q_2)\le \gamma^r(\cF_{l(Q_1)+1})\cdot \mathbf{w},
$$
where the $(i,j)$th coordinate of $\mathbf{w}$ is $La(i,P_{l(Q_1)},Q_1)\cdot La(n-j,P_{l(Q_2)+1},Q_2).$ Clearly, we have $$\gamma^r(\cF_{l(Q_1)+1})\cdot \mathbf{w}\le n^2\max_{i,j}\gamma^r_{i,j}(\cF_{l(Q_1)+1})w(i,j)\le n^{2r+2}\max_{i,j}\gamma^r_{i,j}\left(\binom{[n]}{\lfloor(i+j)/2\rfloor}\right)w(i,j),$$
where the last inequality follows from \tref{gamma}. We can calculate  $\gamma^r_{i,j}(\binom{[n]}{\lfloor(i+j)/2\rfloor})$ by picking the union of size $j$ and the intersection of size $i$ first, which finishes the proof of the upper bound of part \textbf{(a)}. 

To see the furthermore part suppose that the above maximum is obtained when $i$ takes the value $i^*$ and $j$ takes the value $j^*$. We know that there exist two families $\cF_{1,i^*}\subseteq 2^{[i^*]}$ and $\cF_{2,n-j^*}\subseteq 2^{[n-j^*]}$, both unions of full levels, integers $k_1,k_2$ and constants $C_1,C_2$ such that $C_1(i^*)^{k_1}c(Q_1,\cF_{1,i^*})\ge La(i^*,P_{l(Q_1)}Q_1)$ and $C_2(j^*)^{k_2}c(Q_2,\cF_{2,n-j^*})\ge La(n-j^*,P_{l(Q_2)},Q_2)$ hold. Therefore by the upper bound already proven, we know that $La(n,P_{l(Q_1\otimes_r Q_2)+1},Q_1\otimes_r Q_2)$ is at most $n^{2r+2}C_1(i^*)^{k_1}C_2(j^*)^{k_2}\gamma^r_{i^*j^*}(\binom{n}{\lfloor (i^*+j^*)/2\rfloor})La(i^*,P_{l(Q_1)},Q_1)La(n-i^*,P_{l(Q_2)},Q_2)$.

If $\cF_{1,i^*}$ consists of levels of set sizes $h_1,\dots, h_{l(Q_1)}$ and $\cF_{2,n-i^*}$ consists of levels of set sizes $h'_1,\dots, h'_{l(Q'_2)}$, then for the family 
$$\cF:=\binom{[n]}{h_1}\cup \dots \binom{[n]}{h_{l(Q_1)}} \cup \binom{[n]}{\lfloor (i^*+j^*)/2\rfloor}\cup \binom{[n]}{j^*+h'_1}\cup \dots \cup \binom{[n]}{j^*+h'_{l(Q_2)}}$$
we have $c(Q_1\otimes Q_2, \cF)\ge \gamma^r_{i^*j^*}(\binom{n}{\lfloor (i^*+j^*)/2\rfloor})La(i^*,P_{l(Q_1)},Q_1)La(n-i^*,P_{l(Q_2)},Q_2)$. Therefore with $k=2r+2+k_1+k_2$ the family $\cF$ shows that \cjref{genconj} almost holds for the pair $P_{l(Q_1\otimes_r Q_2)+1},Q_1\otimes_r Q_2$.

If $r=1$, then $\cup F_1=\cap F_1=F_1$, so (\ref{lala}) becomes 
$$
\sum_{F\in \cF_{l(Q_1)+1}}La(|F|,P_{l(Q_1)},Q_1)\cdot La(n-| F|,P_{l(Q_2)+1},Q_2).
$$
We can apply \cref{lchain} with $l=k=1$ and $w(i)=La(i,P_{l(Q_1)},Q'_1)\cdot La(i,P_{l(Q_2)},Q'_2)$ to obtain 
$$La(n,P_{l(Q_1\otimes Q_2)+1},Q_1\otimes Q_2)\le \max_{0\le i\le n}\left\{\binom{n}{i}La(i,P_{l(Q_1)},Q'_1)La(n-i,P_{l(Q_2)},Q'_2)\right\}$$
as required.

As the proofs are almost identical we only show the 'strongly holds' case of the furthermore part of \textbf{(b)}. Suppose that the above maximum is obtained when $i$ takes the value $i^*$. We know that there exist two families $\cF_{1,i^*}\subseteq 2^{[i^*]}$ and $\cF_{2,n-i^*}\subseteq 2^{[n-i^*]}$, both unions of full levels, such that $c(Q_1,\cF_{1,i^*})=La(i^*,P_{l(Q_1)}Q_1)$ and $c(Q_2,\cF_{2,n-i^*})=La(n-i^*,P_{l(Q_2)},Q_2)$ hold. If $\cF_{1,i^*}$ consists of levels of set sizes $j_1,\dots, j_{l(Q_1)}$ and $\cF_{2,n-i^*}$ consists of levels of set sizes $j'_1,\dots, j'_{l(Q'_2)}$, then for the family 
$$\cF:=\binom{[n]}{j_1}\cup \dots \binom{[n]}{j_{l(Q_1)}} \cup \binom{[n]}{i^*}\cup \binom{[n]}{i^*+j'_1}\cup \dots \cup \binom{[n]}{i^*+j'_{l(Q_2)}}$$
we have $c(Q_1\otimes Q_2, \cF)=\binom{n}{i^*}La(i^*,P_{l(Q_1)},Q_1)La(n-i^*,P_{l(Q_2)},Q_2)$.
\end{proof}

\begin{proof}[Proof of \tref{ragasztas2}]
The proof goes very similarly to the proof of \tref{ragasztas1}. Let us consider the canonical partition of a $P_{l(Q\otimes r)+1}$-free family $\cF\subseteq 2^{[n]}$. Then in any copy of $Q\otimes r$ in $\cF$, if $F_1,\dots,F_r$ correspond to the $r$ top elements forming an antichain, we must have $F_1,\dots, F_r \in \cF_{l(Q)+1}$. Also, if a copy of $Q\otimes r$ contains $F_1,\dots, F_r$, then the sets corresponding to the other elements of the poset must be contained in $\bigcap_{l=1}^r F_l$. Let $j=|\bigcap_{l=1}^r F_l|$. Then the number of copies of $Q\otimes r$ in $\cF$ that contain $F_1, \dots, F_r$ is at most $La(j,P_{l(Q)+1},Q)$. If $r\ge 2$, we obtain
$$
c(Q\oplus r,\cF)\le \beta^r(\cF_{l(Q)+1})\cdot \mathbf{w},
$$
where the $j$th coordinate of $\mathbf{w}$ is $La(j,P_{l(Q)+1},Q)$. Clearly we have $\beta^r(\cF_{l(Q)+1})\cdot \mathbf{w}\le n \max_{i}\beta_i^r(\cF_{l(Q)+1})w(i)\le n^{2r+2}\max_i \beta_i^r(\binom{[n]}{j(r,i,n)})w(i),$ where the last inequality follows from \tref{intuni}. We have $\beta_i^r(\binom{[n]}{j(r,i,n)})=\binom{n}{i}\beta_0^r(\binom{[n-i]}{j(r,i,n)-i})$ by picking the intersection of size $i$ first. 

To see the furthermore part of \textbf{(a)}, let $i^*$ be the value of $i$ for which the above maximum is attained. Then if $\cF_{i^*}=\binom{[i^*]}{h_1}\cup \dots \cup \binom{[i^*]}{h_{l(Q)}}$ is a family with $Cn^kc(Q,\cF_{i^*})\ge La(i^*,P_{l(Q)+1},Q)$, then for the family $\cF^*=\binom{[n]}{h_1}\cup \dots \cup \binom{[n]}{h_{l(Q)}}\cup \binom{[n]}{j(r,i^*,n)}$, we have 
\[
c(Q\oplus r, \cF^*)\ge \beta^r_{i^*}\left(\binom{[n]}{j(r,i^*,n)}\right)c(Q,\cF_{i^*})\ge \binom{n}{i^*}\beta^r_{0}\left(\binom{[n-i^*]}{j(r,i^*,n)-i^*}\right)\frac{1}{Cn^k}La(n,i^*,Q),
\]
therefore $\cF^*$ with $C'=C$ and $k'=2r+2+k$ shows that \cjref{genconj} almost holds for the pair $P_{l(Q)+2},Q\oplus r$.

If $r=1$, then $|\cap F_1|=|F_1|$, so applying \cref{lchain} with $l=k=1$ we obtain
\[
c(Q\oplus 1, \cF)\le \sum_{F\in\cF_{l(Q)+1}}La(|F|,P_{l(Q)+1},Q)\le \max_{0\le i\le n}\left\{\binom{n}{i}La(i,P_{l(Q)+1},Q)\right\}. 
\]

The proof of the furthermore part of \textbf{(b)} is analogous to the previous ones and is left to the reader.
\end{proof}

\begin{proof}[Proof of \cref{genmult}]
We proceed by induction on the number of levels. The base case is guaranteed by Sperner's \tref{sperner}. The inductive step follows by applying \tref{ragasztas2} as $K_{r_1,\dots,r_l}=K_{r_1,\dots, r_{l-1}}\oplus r_l$.
\end{proof}

\begin{proof}[Proof of \cref{genmult1}]
We proceed by induction on the number of levels. The base case is guaranteed by Sperner's \tref{sperner}. Suppose the statement has been proved for all complete multipartite posets satisfying the condition with height smaller than $l$ and consider $K_{r_1,r_2,\dots,r_l}$. We know that there exists an $i$ with $1\le i \le l$ such that $r_i=1$. If $1<i<l$, then the inductive step will follow by applying the furthermore part of \tref{ragasztas1} to $Q_1=K_{r_1,\dots, r_{i-1}}$ and $Q_2=K_{r_{i+1},\dots, r_l}$. If $i=l$, then the inductive step will follow by applying the furthermore part of \tref{ragasztas2} to $Q=K_{r_1,\dots, r_{l-1}}$ and $r=1$. The case $i=1$ follows from $c(K_{r_1,\dots,r_l},\cF)=c(K_{r_l,\dots,r_1},\overline{\cF})$, where $\overline{\cF}=\{[n]\setminus F:F\in\cF\}$.
\end{proof}

\begin{proof}[Proof of \tref{multi1}]
We only give the sketch of the proof as it is very similar to previous ones. Consider a $P_{l+3}$-free family $\cF\subseteq 2^{[n]}$ and its canonical partition. If $l=1$, then we count the number of copies of $K_{r,1,s}$ according to the set $F\in \cF_2$ that plays the role of the middle element of $K_{r,1,s}$. The number of copies that contain $F$ is not more than $\binom{\binom{|F|}{\lfloor |F|/2\rfloor}}{r}\binom{\binom{n-|F|}{\lfloor (n-|F|)/2\rfloor}}{s}$. Applying \cref{lchain} with $l=k=1$ and $w(i)=\binom{\binom{i}{\lfloor i/2\rfloor}}{r}\binom{\binom{n-i}{\lfloor (n-i)/2\rfloor}}{s}$ yields $c(K_{r,1,s},\cF)\le \max_i\binom{n}{i}\binom{\binom{i}{\lfloor i/2\rfloor}}{r}\binom{\binom{n-i}{\lfloor (n-i)/2\rfloor}}{s}$. Let $i^*$ be the value of $i$ for which this maximum is attained. Then the family $\binom{[n]}{\lfloor i^*/2\rfloor}\cup \binom{[n]}{i^*}\cup \binom{[n]}{\lfloor (n+i^*)/2\rfloor}$ contains exactly that many copies of $K_{r,1,s}$.

If $l\ge 2$, then we count the number of copies of $K_{r,1,1,\dots,1,s}$ according to the sets $F_2\in \cF_2$ and $F_{l+1}\in \cF_{l+1}$ playing the role of the elements on the second and $(l+1)$st level of $K_{r,1,1,\dots,1,s}$. For a fixed pair $F_2\in \cF_2$ and $F_{l+1}\in \cF_{l+1}$ with $F_2 \subset F_{l+1}$ the number of copies of $K_{r,1,1,\dots,1,s}$ containing $F_2$ and $F_{l+1}$ is at most $\binom{\binom{|F_2|}{\lfloor |F_2|/2\rfloor}}{r}\binom{\binom{n-|F_{l+1}|}{\lfloor (n-|F_{l+1}|)/2\rfloor}}{s}La(|F_{l+1}|-|F_2|,P_{l-1},P_{l-2})$. The value of $La(|F_{l+1}|-|F_2|,P_{l-1},P_{l-2})$ is given by \tref{gp}. So we can apply \cref{lchain} with $l=k=2$ and $w(i,j)=\binom{\binom{i}{\lfloor i/2\rfloor}}{r}\binom{\binom{n-j}{\lfloor (n-j)/2\rfloor}}{s}La(j-i,P_{l-1},P_{l-2})$ to obtain $c(K_{r,1,1\dots,1,s},\cF)\le \max_{i,j}\binom{n}{j}\binom{j}{i}\binom{\binom{i}{\lfloor i/2\rfloor}}{r}\binom{\binom{n-j}{\lfloor (n-j)/2\rfloor}}{s}La(j-i,P_{l-1},P_{l-2})$. Let $i^*$ and $j^*$ be the values of $i$ and $j$ for which this maximum is attained. Then the family consisting of $\binom{[n]}{\lfloor i^*/2\rfloor}, \binom{[n]}{i^*},\binom{[n]}{j^*}, \binom{[n]}{\lfloor (n+j^*)/2\rfloor}$ and the $l-2$ full levels determined by \tref{gp} contains exactly that many copies of $K_{r,1,1,\dots,1,s}$.
\end{proof}




\vskip 0.3truecm

There are several other complete multi-partite posets for which one can determine the levels that form an almost optimal family. For example, using \tref{intuni} \textbf{(a)} one can prove that $La(n,P_3,K_{p,2})\le nc(K_{p,2},\cF)$ where $\cF=\binom{[n]}{i}\cup \binom{[n]}{j}$ with $i=(\frac{2^{p-1}}{3+2^p}+o(1))n$ and $j=(\frac{1+2^p}{3+2^p}+o(1))n$. In particular $La(n,P_3,K_{p,2})=2^{(c_p+o(1))n}$, where $c_p=\frac{2+p\cdot 2^p}{3+2^p}+h(\frac{2^p}{3+2^p})+\frac{3}{3+2^p}h(2/3)$.

\begin{figure}[h]
\centering
\includegraphics{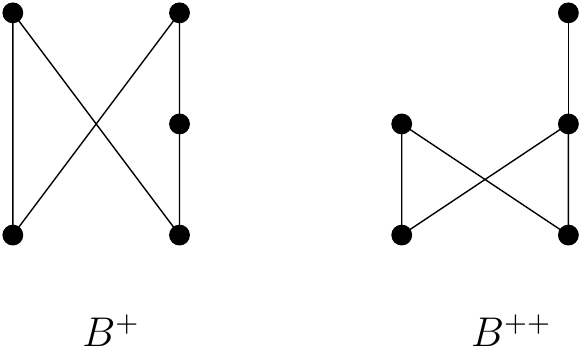}
\caption{The Hasse diagrams of the posets $B^+$ and $B^{++}$}
\label{fig:defs2}
\end{figure}

Let us finish this section by some remarks about $K_{2,2}=B$ as there exist several extremal results concerning $B$. Let us consider the following two posets that contain $B$:  $B^+$ and $B^{++}$ have five elements $a,b,c,d,e$ such that $a<_{B^+}c,e$ and $b<_{B^+} c,d$ and also $d<_{B^+}e$, while $a,b<_{B^{++}}c,d$ and $d<_{B^{++}}e$. By results of DeBonis, Katona, and Swanepoel \cite{DKS} and Methuku and Tompkins \cite{MT} we know that $La(n,B,P_1)=La(n,B^+,P_1)={n \choose \lfloor n/2\rfloor}+{n \choose \lfloor n/2\rfloor+1}$, and as $B^{++}$ contains a copy of $B^+$ we have $La(n,B^+,P_1)\le La(n,B^{++},P_1)$. It is natural to ask how many copies of $B$ can a $B^+$-free or $B^{++}$-free family in $2^{[n]}$ contain, especially that the largest $B^+$-free family does not contain any. Obviously, a $P_3$-free poset is both $B^+$-free and $B^{++}$-free, therefore we obtain the inequality $La(n,P_3,B)\le La(n,B^+,B)\le La(n,B^{++},B)$. The next proposition shows that these functions are asymptotically equal.

\begin{proposition}\label{prop:BB}
$$
La(n,P_3,B)\le La(n,B^+,B) \le La(n,B^{++},B)=(1+o(1))La(n,P_3,B).$$
\end{proposition}

\begin{proof}

\vskip 0.2truecm

The first two inequalities are true by definition. Note that $La(n,P_3,B)=2^{(c_2+o(1))n}$, where $c_2=10/7+h(4/7)+3h(2/3)/7$ by the remarks made after the proof of \tref{multi1}.

Let $\cF\subseteq 2^{[n]}$ be $B^{++}$-free, and consider its canonical partition (note that $P_5$ contains a copy of $B^{++}$, thus $\cF$ is $P_5$-free).

Consider first the family $\cS \subseteq \cF_2$ of sets that appear in 4-chains in $\cF$ (they must be the second smallest in those chains). Note that if $S\in \cS$ with $F_1\subsetneq S \subsetneq F_3 \subsetneq F_4$, then $S$ is not comparable to any other set $F$ of $\cF$ as $F,F_1,S,F_3,F_4$ would form a copy of $B^{++}$ both if $S\subsetneq F$ or $F \subsetneq S$. Therefore every set $S\in \cS$ is contained in at most one copy of $B$ in $\cF$. As $\cS\subseteq \cF_2$ is an antichain, we obtain $c(B,\cF)-c(B,\cF\setminus \cS)\le \binom{n}{n/2}$.

Clearly, $\cF'=\cF\setminus S$ is $P_4$-free. Let us consider its canonical partition and denote the resulting antichains by $\cF'_1, \cF'_2, \cF'_3$. Let $\cS'\subseteq \cF'_2$ be the family of middle sets of all 3-chains in $\cF'$. We know that for any $S\in \cS'$ there exist $F'_1,F'_3\in \cF'$ with $F'_1\subsetneq S\subsetneq F'_3$. Also, there cannot exist $F''_1,F''_3$ with $F'_1\subsetneq S\subsetneq F'_3$ as then $F'_1,F''_1,S,F'_3,F''_3$ would form a copy of $B^{++}$. So either there is a unique set $F'$ that contains $S$ and potentially several sets that are contained in $S$ or there exists a unique $F'$ contained in $S$ and several sets containing $S$. In the former case, if $S$ is contained in a copy of $B$, it can only be one of the top sets. Furthermore, if a copy of $B$ contains $S$, then it contains $F'$ as otherwise this copy could be extended by $F'$ to form a $B^{++}$. As the sets contained in $S$ form an antichain (they are a subfamily of $\cF'_1$), we obtain that the number of copies of $B$ containing $S$ is at most $\binom{\binom{|S|}{\lfloor\frac{|S|}{2}\rfloor}}{2}$. Similarly, if $S$ contains exactly one other set of $\cF'$, then the number of copies of $B$ containing $S$ is at most $\binom{\binom{n-|S|}{\lfloor\frac{n-|S|}{2}\rfloor}}{2}$. So introducing $w(i)=\max\{\binom{\binom{i}{\lfloor i/2\rfloor}}{2},\binom{\binom{n-i}{\lfloor(n-i)/2\rfloor}\}}{2}\}$ we obtain that the total number of copies containing at least one element of $\cS'$ is at most $\sum_{S\in \cS'}w(|S|)$. By the special case $k=l=1$ of \cref{lchain} we obtain that this expression is maximized over all antichains when $\cS$ is a full level of $2^{[n]}$.

The weight function $w$ is symmetric, i.e. $w(i)=w(n-i)$ holds for any $i$, therefore it is enough to maximize $\binom{n}{i}w(i)$ over $n/2\le i \le n$. It is a routine exercise to see that $\binom{n}{i}\binom{\binom{i}{\lfloor i/2\rfloor}}{2}$ is maximized when $i=(4/5+o(1))n$. Therefore the number of copies of $B$ that contain an element of $S'$ is at most $2^{h(4/5)+8/5+o(1)}$. We obtained that
\begin{equation*}
\begin{split}
c(B,\cF) & \le c(B,\cF\setminus (S\cup S'))+ \binom{n}{\lfloor n/2\rfloor}+2^{h(4/5)+8/5+o(1)} \\
& \le La(n,P_3,B)+\binom{n}{\lfloor n/2\rfloor}+2^{h(4/5)+8/5+o(1)}\\
& = (1+o(1))La(n,P_3,B),
\end{split}
\end{equation*}
as $h(4/5)+8/5=2.3219...<c_2$.
\end{proof}

\section{Remarks}

One can define an even more general parameter $La_R(P,Q)$. For three posets, $R,P$ and $Q$ we are interested in the maximum number of copies of $Q$ in subposets $R'$ of $R$ that do not contain $P$. Analogously to what we had for set families, we say that $R'\subseteq R$ is \textit{a copy of $Q$ in $R$} if there exists a bijection $\phi:Q\rightarrow R'$ such that whenever $x \le_Q x'$ holds, then so does $\phi(x)\le _{R'} \phi(x')$. Let $c(Q,R)$ denote the number of copies of $Q$ in $R$ and for any three posets $R,P$ and $Q$ we define
\[
La_R(P,Q)=\max\{c(Q,\cF): R'\subseteq R, c(P,R')=0 \},
\]
and for a poset $R$ and families of posets $\cP,\cQ$ let us define
\[
La_R(\cP,\cQ)=\max\left\{\sum_{Q\in\cQ}c(Q,R'):R'\subseteq R, \forall P\in\cP  \hskip 0.2truecm c(P,R')=0\right\}.
\]
Note that $La(n,P,Q)=La_{B_n}(P,Q)$, where $B_n$ is the poset with elements of $2^{[n]}$ ordered by inclusion.
Very recently Guo, Chang, Chen, and Li \cite{GCCL} introduced $La_R(Q,P_1)$, as a general approach to forbidden subposet problems. That is to solve the analogous question in a less complicated structure like the cycle, chain or double chain, and then to apply an averaging argument.

\bigskip
In many parts of \tref{easy}, the construction yielding the lower bound that matches the upper bound contained the empty set and/or the set $[n]$. One might wonder whether the $La$-function remains the same if we do not allow these elements to be included. In other words, if $B_n^-$ denotes the subposet of $B_n$ with $\emptyset$ and $[n]$ removed, then how $La_{B_n^-}(\vee, P_2)$ relates to $La(n,\vee, P_2)$, $La_{B_n^-}(B, P_3)$ to $La(n,B, P_3)$ and so on. The $\{\bigvee,\bigwedge\}$-free construction $\binom{[n-1]}{\lfloor(n-1)/2\rfloor} \cup \{F\cup\{n\}: F \in \binom{[n-1]}{\lfloor(n-1)/2\rfloor} \}$ and the $B$-free construction $\binom{[n-2]}{\lfloor (n-2)/2\rfloor}\cup \{F\cup \{n-1\}: F\in \binom{[n-2]}{\lfloor (n-2)/2\rfloor} \}\cup \{F\cup \{n-1,n\}: F\in \binom{[n-2]}{\lfloor(n-2)/2\rfloor} \}$ 
show that

\begin{itemize}
\item
${n-1 \choose \lfloor (n-1)/2\rfloor}\le La_{B_n^-}(\bigvee,P_2)=La_{B_n^-}(\bigwedge,P_2)\le {n\choose \lfloor n/2\rfloor}$,
\item
${n-2 \choose \lfloor(n-2)/2\rfloor}\le La_{B_n^-}(B,P_3)\le {n \choose \lfloor n/2\rfloor}$.
\end{itemize}
There is a longstanding (folklore) conjecture which would imply the existence of constructions in both cases that asymptotically match the upper bounds. Let $\cM_{k+1}\subseteq \binom{[n]}{k+1}$ be a family of sets with the property that for every $K \in \binom{[n]}{k}$ there exists at most one set $M\in \cM_{k+1}$ with $K\subsetneq M$. Obviously, for any such set we have $|\cM_{k+1}|\le \binom{n}{k}/(k+1)$ and $\cR_{k+1}:=\cM_{k+1}\cup \binom{[n]}{k}$ is $\bigvee$-free with $c(P_2,\cR_{k+1})=(k+1)|M_{k+1}|$. It is conjectured that there exists a family $\cM_{\lfloor n/2\rfloor+1}$ with the above property such that $|\cM_{\lfloor n/2\rfloor+1}|=(1-o(1))\binom{[n]}{\lfloor n/2\rfloor}/(\lfloor n/2\rfloor+1)$ holds. Similarly, writing $\overline{\cM}_{n-k+1}$ for $\{[n]\setminus M:M \in \cM_{n-k+1}\}$ the construction $\cT_k:=\cM_{k+1}\cup \binom{[n]}{k}\cup \overline{\cM}_{n-k+1}$ is $B$-free. The above conjecture would yield $c(P_3,\cT_{\lfloor n/2\rfloor})=(1-o(1))\binom{n}{\lfloor n/2\rfloor}$. 

\bigskip
In Section 4, we proved that apart from a polynomial factor  \cjref{genconj} holds for complete multi-level posets, i.e. there exists a sequence $\cF_n$ of families that consists of full levels such that $La(n,P_{l+1},K_{r_1,r_2,\dots,r_l})\le n^kc(\cF_n, K_{r_1,r_2,\dots,r_l})$ for some constant $k=k(K_{r_1,r_2,\dots,r_l})$. To improve this result or to completely get rid of the polynomial factor one would need to improve \tref{rtuples} or rather to determine the intersection profile polytope of antichains.

\end{document}